\documentclass[11pt]{amsart}

\calclayout

\usepackage{color}
\usepackage{amsfonts}
\usepackage{amssymb}
\usepackage{amsmath}
\usepackage{hyperref}
\usepackage{multicol}

\usepackage{graphicx}
\usepackage{enumitem}
\usepackage{subcaption}

\newcommand{\La}{\langle}
\newcommand{\Ra}{\rangle}

\newcommand{\pd}{\partial}

\def\supp{\operatorname{supp}}

\newcommand{\al}{\alpha}

\newcommand{\la}{\lambda}
\newcommand{\vf}{\varphi}

\newcommand{\cB}{\mathcal B}
\newcommand{\cD}{\mathcal D}

\newcommand{\bB}{\mathbb B}

\newcommand{\bR}{\mathbb R}

\newtheorem{theorem}{Theorem}[section]
\newtheorem{cor}[theorem]{Corollary}
\newtheorem{lemma}[theorem]{Lemma}

\theoremstyle{definition}
\newtheorem{remark}[theorem]{Remark}

\numberwithin{equation}{section}

\newcounter{vremennyj}

\begin{document}

\title[]%
{
 Bellman function sitting  on a tree
}
\author[N.~Arcozzi]{Nicola Arcozzi}
\address{Universit\`{a} di Bologna, Department of Mathematics, Piazza di Porta S. Donato, 40126 Bologna (BO)}
\email{nicola.arcozzi@unibo.it}
\thanks{Theorem 3.1 was obtained in the frameworks of the project 17-11-01064 by the Russian Science Foundation}
\thanks{NA is partially supported by the grants INDAM-GNAMPA 2017 "Operatori e disuguaglianze integrali in spazi con simmetrie" and PRIN 2018 "Variet\`{a} reali e complesse: geometria, topologia e analisi armonica"}
\author[I.~ Holmes]{Irina Holmes}
\thanks{IH is partially supported by the NSF an NSF Postdoc under Award No.1606270}
\address{Department of Mathematics, Michigan Sate University, East Lansing, MI. 48823}
\email{holmesir@msu.edu}
\author[P. Mozolyako]{Pavel Mozolyako}
\thanks{PM is supported by the Russian Science Foundation grant 17-11-01064}
\address{Universit\`{a} di Bologna, Department of Mathematics, Piazza di Porta S. Donato, 40126 Bologna (BO)}
\email{pavel.mozolyako@unibo.it}
\author[A.~Volberg]{Alexander Volberg}
\thanks{AV is partially supported by the NSF grant DMS-160065}
\address{Department of Mathematics, Michigan Sate University, East Lansing, MI. 48823}
\email{volberg@math.msu.edu \textrm{(A.\ Volberg)}}
\makeatletter
\@namedef{subjclassname@2010}{
  \textup{2010} Mathematics Subject Classification}
\makeatother
%
%
\begin{abstract}
In this note  we give a  proof-by-formula of  certain  important embedding inequalities on a tree. We also consider the case of a bi-tree, where a different
approach is explained.
\end{abstract}
\maketitle

\subsection{Hardy operator on a tree}
\label{Hop}
Let $I^0$ be a unit interval. Let us associate the dyadic lattice $\mathcal{D}(I^0)$ and  the uniform directed dyadic tree $T$ in a usual way. First we define the Hardy operator, the dual Hardy operator and the logarithmic potential: given a function $\varphi: T\rightarrow \mathbb{R}_+$ we let
\begin{equation}\notag
\begin{split}
& (I\varphi)(\alpha) = \sum_{\beta\geq\alpha}\varphi(\beta),\quad \alpha \in T;\\
&(I^*\varphi)(\alpha) = \sum_{\beta\leq\alpha}\varphi(\beta),\quad \beta \in T
;\\
& V^{\varphi}(\gamma) = (II^*\varphi)(\gamma),\quad \gamma\in T,
\end{split}
\end{equation}
where $\leq$ is the natural order relation on $T$.

We always may think that the tree $T$ is finite (albeit very large). By the boundary $\pd T$ we understand the vertices that are not 
connected to smaller vertices. 

Each dyadic interval $Q$ in $\mathcal{D}(I^0)$ corresponds naturally to a vertex $\al_Q$.

Let $\mu$ be a measure on the  tree $T$, so just the collection of non-negative numbers $\{\mu_P\}_{P\in T}$.
Assuming $\mu$ to be a measure on $T$, we have
\begin{equation}\notag
\begin{split}
& (I\mu)(\al_R) = \sum_{Q\supset R}\mu_Q,\quad Q,R\in \mathcal{D}(I^0);\\
&(I^*\mu)(\al_Q) = \mu(Q)=\sum_{P\subset Q, \al_P\in \pd T} \mu_P,\quad Q \in \mathcal{D}(I^0);\\
& V^{\mu}(\al_P) = (II^*\mu)(\al_P),\quad P\in \mathcal{D}(I^0),
\end{split}
\end{equation}
the second equality is valid under the  assumption of $\supp\mu\subset  \partial T$. 


We will answer the question when $I : \ell^2(T) \to \ell^2(T, \mu)$.  Passing to the adjoint operator we see that this is equivalent to the following inequality
\begin{equation}
\label{adj}
\sum_{Q\in T} \big(\sum_{P \le Q} \vf (P)\mu_P\big)^2  \lesssim \big(\sum_{R\in T} \vf(R)^2 \mu_R\big)\,.
\end{equation}

\begin{theorem}
\label{TREE1}
Operator $I$  is a bounded operator $I : \ell^2(T) \to \ell^2(T, \mu)$ if and only if 
\begin{equation}
\label{Carl0}
\sum_{Q\in T, \, Q\le R} \big(\sum_{P \le Q} \mu_P\big)^2  \lesssim \big(\sum_{Q \le R} \mu_Q\big)\quad \forall R\in T\,.
\end{equation}
\end{theorem}

This is proved in Theorem \ref{T:Carleson embedding theorem-tree}  below by the use of Bellman function.

\section{Bellman function on a tree}
\label{Btree}

\begin{theorem}
\label{tree1}
Let $dw$ be a positive measure on $I_0:=[0,1]$. Let $\La w\Ra_I$ denote $w(I)/|I|$.
Let $\varphi$ be a measurable test function. Then if 
	\begin{equation}
	\label{test1}
	\frac{1}{|J|} \sum_{I\in \cD(J)} \La w\Ra_I^2 |I|^2 \le \La w\Ra_J\quad \forall J\in \cD(I_0),
	\end{equation}
then
	\begin{equation}
	\label{emb1}
	\sum_{I\in \cD(I_0)} \La \varphi w\Ra_I^2 |I|^2 \lesssim \La \varphi^2 w\Ra_{I_0} |I_0|,
	\end{equation}
\end{theorem}

\noindent This can be obtained as a direct consequence of the weighted Carleson embedding theorem \cite{NTV99}:

\begin{theorem}
\label{T:W-CET}
Let $\cD$ be a dyadic lattice, $w$ be any weight, and $\{\alpha_I\}_{I\in\cD}$ be a sequence of non-negative numbers. Then, if
	\begin{equation}\label{E:wCarleson embedding theoremtest}
	\frac{1}{|J|} \sum_{I \subset J} \alpha_I \La w\Ra_I^2 \leq \La w\Ra_{J} \quad \forall J\in \cD,
	\end{equation}
then
	\begin{equation}\label{E:wCarleson embedding theoremconc}
	\sum_{I\in\cD} \alpha_I \La \varphi \sqrt{w}\Ra_{I}^2 \lesssim \|\varphi\|_{L^2}^2, 
	\end{equation}
for all $\varphi \in L^2$.
\end{theorem}

Clearly, the conclusion of \eqref{E:wCarleson embedding theoremconc} may be rewritten as
	$$\frac{1}{|I_0|} \sum_{I\subset I_0} \alpha_I \La \varphi w\Ra_I^2 \lesssim \La \varphi^2 w\Ra_{I_0}. $$
Letting $\alpha_I = |I|^2$ in \eqref{E:wCarleson embedding theoremtest}, we obtain exactly Theorem \ref{tree1}.

We recall here that the proof of Theorem \ref{T:W-CET} in \cite{NTV99} was based upon the Bellman function
	\begin{equation}\label{E:BellmanDef}
	\cB(F, f, A, v) := 4\bigg(F - \frac{f^2}{v + A}\bigg),
	\end{equation}
and three main properties this function satisfies are:
	\begin{enumerate}
	\item $\cB$ is defined on:
		$$f^2 \leq Fv; A \leq v;$$
	\item $0 \leq \cB \leq CF$, in this case with $C = 4$;
	\item Main Inequality: 
		\begin{equation}\label{MI99}
		\cB(F, f, A, v) - \frac{1}{2}\bigg(\cB(F_{-}, f_{-}, A_{-}, v_{-}) + \cB(F_{+}, f_{+}, A_{+}, v_{+}) \bigg) \geq \frac{f^2}{v^2}m,
		\end{equation}
	for all points in the domain such that
		\begin{equation}
		\label{mart}
		F = \frac{F_{-}+ F_{+}}{2}; \:\: f=\frac{f_{-}+f_{+}}{2}; \:\: v = \frac{v_{-}+v_{+}}{2},
		\end{equation}
	and 
		$$A = m + \frac{A_{-}+A_{+}}{2},$$
	for some $m\geq 0$.
	\end{enumerate}
In particular, we have that the function $\cB$ is \textit{concave}.

\subsection{Carleson embedding theorem on a dyadic tree}
\label{Tree}

Now we wish to prove a version of Theorem \ref{tree1} on a dyadic tree. Specifically, suppose we have a dyadic tree originating at some $I_0 \in \cD$. Define a measure $\Lambda$ on the tree as follows: to each node $I \in \cD(I_0)$ we associate a non-negative number $\lambda_I \geq 0$. We may think of $I \in \cD(I_0)$ as an interval \textit{in the dyadic tree} by considering $\{K \in \cD(I_0): K \subset I\}$. Then we define
	$$\Lambda(I) := \sum_{K\subset I} \lambda _K,$$
and the averaging operator
	$$(\Lambda)_I := \frac{1}{|I|}\Lambda(I).$$
Given a function $\varphi = \{\varphi(I)\}_{I \in \cD(I_0)}$ on the dyadic tree, we have
	$$\int_I \varphi \,d\Lambda = \sum_{K \subset I} \varphi(K)\lambda_K,$$
and
	$$(\varphi\Lambda)_I := \frac{1}{|I|}\int_I \varphi\,d\Lambda.$$

\begin{theorem}[Carleson embedding theorem for a dyadic tree]\label{T:Carleson embedding theorem-tree}
Let $I_0\in\cD$, the dyadic tree originating at $I_0$ with notations as above, and $\{\alpha_I\}_{I\subset I_0}$ be a sequence of non-negative numbers. Then, if
	\begin{equation}\label{E:wCarleson embedding theorem-tree-test}
	\frac{1}{|I|} \sum_{K\subset I} \alpha_K (\Lambda)_K^2 \leq (\Lambda)_I, \quad \forall I \in \cD(I_0),
	\end{equation}
then
	\begin{equation}\label{E:wCarleson embedding theorem-tree-conc}
	\frac{1}{|I_0|} \sum_{I\subset I_0} \alpha_I (\varphi\sqrt{\Lambda})_I^2 \leq 4 (\varphi^2)_{I_0},
	\end{equation}
where
	$$(\varphi\sqrt{\Lambda})_I := \frac{1}{|I|}\sum_{K\subset I} \varphi(K)\sqrt{\lambda_K} \:\:\text{ and }\:\:
		(\varphi^2)_{I_0} := \frac{1}{|I_0|} \sum_{I\subset I_0} \varphi(I)^2. $$
\end{theorem}

Note that the conclusion of \eqref{E:wCarleson embedding theorem-tree-conc} may be rewritten as
	$$\sum_{I\subset I_0} \alpha_I (\varphi\Lambda)_I^2 \leq 4 (\varphi^2 \Lambda)_{I_0}.$$
Letting $\alpha_I = |I|^2$ in \eqref{E:wCarleson embedding theorem-tree-test}, we obtain:

\begin{cor}\label{C:tree2}
Let $I_0\in\cD$, the dyadic tree originating at $I_0$ with notations as above. Then, if
	$$\frac{1}{|I|} \sum_{K\subset I} |K|^2 (\Lambda)_K^2 \leq (\Lambda)_I, \quad \forall I \in \cD(I_0),$$
then
	$$\frac{1}{|I_0|} \sum_{I\subset I_0} |I|^2 (\varphi\Lambda)_I^2 \leq 4 (\varphi^2 \Lambda)_{I_0}.$$
\end{cor}

The proof of Theorem \ref{T:Carleson embedding theorem-tree} is based also on the function $\cB$ in \eqref{E:BellmanDef}, and on proving a more involved version of \eqref{MI99} -- this will be Lemma \ref{L:BEstimate}. 

For now, let us create a Bellman function for the dyadic tree. Below, we have in the left column the setup for the original Bellman function of \cite{NTV99}, and on the right we construct the Bellman function for our Carleson embedding theorem on the dyadic tree.

\begin{multicols}{2}
\setlength{\columnseprule}{0.4pt}

\begin{center}
\textbf{Classic Weighted CET}
\end{center}
$$\mathbb{B}_1(F, f, A, v) := \sup_{\varphi, w, \alpha} \frac{1}{|I_0|}\sum_{I\subset I_0} \alpha_I \La \varphi\sqrt{w}\Ra_I^2$$
where the supremum is over all functions $\varphi$ on $I_0$, weights $w$ on $I_0$, and $w$-Carleson sequences 
$\alpha = (\alpha_I)_{I\subset I_0}$ such that:
	\begin{itemize}
	\item $\La\varphi^2\Ra_{I_0} = F$
	\item $\La\varphi\sqrt{w}\Ra_{I_0} = f$
	\item $\frac{1}{|I_0|}\sum_{I\subset I_0}\alpha_I\La w\Ra_I^2 = A$
	\item $\La w\Ra_{I_0} = v.$
	\end{itemize}

\vfill\null
\columnbreak

\begin{center}
\textbf{CET on Dyadic Tree}
\end{center}
$$\mathbb{B}_2(F, f, A, v) := \sup_{T, \varphi, \Lambda, \alpha} \frac{1}{|I_0|}\sum_{I\subset I_0} \alpha_I  (\varphi\sqrt{\Lambda})_I^2$$
where the supremum is over all dyadic trees $T$ originating at $I_0$, measures $\Lambda$ on $T$, functions $\varphi$ on $T$ and non-negative sequences $\alpha = (\alpha_I)_{I\subset I_0}$ such that:
	\begin{itemize}
	\item $(\varphi^2)_{I_0} = F$
	\item $(\varphi\sqrt{\Lambda})_{I_0} = f$
	\item $\frac{1}{|I_0|}\sum_{I\subset I_0} \alpha_I (\Lambda)_I^2 = A$
	\item $(\Lambda)_{I_0} = v$.
	\end{itemize}
\end{multicols}

Both will be defined on $\{f^2 \leq Fv; A \leq v\}$, and both will satisfy some Main Inequality -- which will turn out to be the fundamental distinction between the two. Let us also mention here that the function $\cB$ in \eqref{E:BellmanDef} is not the ``true'' Bellman function $\bB_1$ above, but a \textit{supersolution}. This refers in this case to any function which satisfies properties (1) -- (3) in Section \ref{Btree}. The true Bellman function $\bB_1$ was found in \cite{VV09}.

Now let us discuss the Main Inequalities for these functions. For $\bB_1$, \eqref{MI99} was obtained in the usual way, by running the Bellman machine separately on each half of an interval $I_0$ -- see Figure \ref{fig:Pic2}(A): choose weights $w_{\pm}$ and functions $\varphi_{\pm}$ supported on $I_0^{\pm}$, and $w_{\pm}$-Carleson sequences $\alpha_{\pm} = (\alpha_I)_{I\subset I_0^{\pm}}$, in such a way that they give the supremum for $\bB_1(F_{\pm}, f_{\pm}, A_{\pm}, v_{\pm})$ up to some $\epsilon$.

\begin{figure}[h!]
  \centering
  \begin{subfigure}[b]{0.45\linewidth}
    \includegraphics[width=\linewidth]{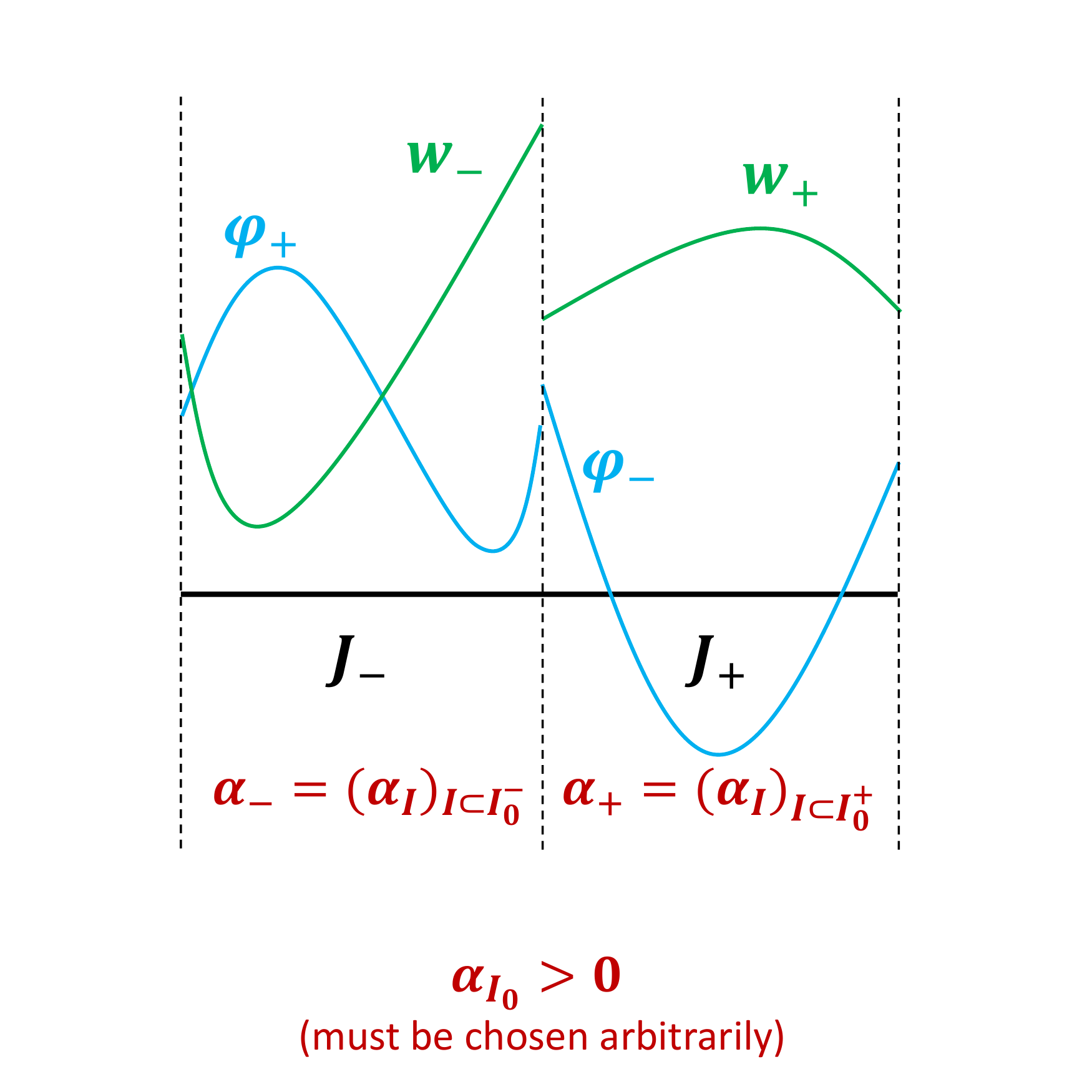}
      \caption{Main Inequality for $\bB_1$ -- classic case}
  \end{subfigure}
  \begin{subfigure}[b]{0.45\linewidth}
    \includegraphics[width=\linewidth]{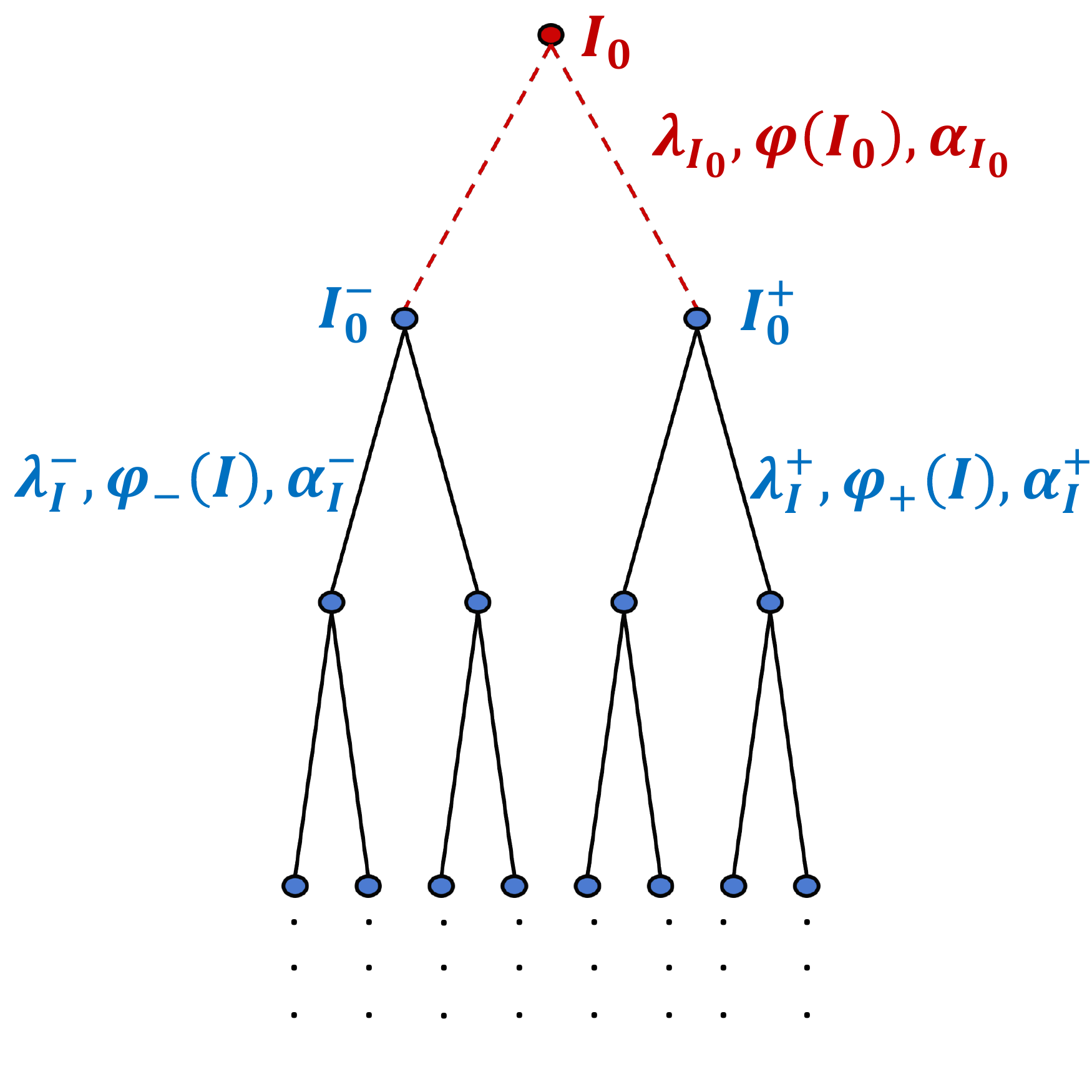}
    \caption{Main Inequality for $\bB_2$ -- dyadic tree case}
  \end{subfigure}
  \caption{Deducing the Main Inequalities for $\bB$}
  \label{fig:Pic2}
\end{figure}

One then easily obtains a weight $w$ and a function $\varphi$ on $I_0$ by concatenation. In the case of the sequence $\alpha$, one must choose an $\alpha_J > 0$ arbitrarily though, as long as the resulting sequence remains $w$-Carleson. This is what produces the $m$ term in \eqref{MI99}.

Now let us turn to $\bB_2$ and proceed similarly: take two dyadic trees $T_{\pm}$ originating at $I_0^{\pm}$, each equipped with measures
$\Lambda_{\pm} = \{\lambda_I^{\pm}\}_{I\subset I_0^{\pm}}$, two function $\varphi_{\pm} = \{\varphi_{\pm}(I)\}_{I\subset I_0^{\pm}}$ on the trees, and non-negative sequences $\alpha_{\pm} = \{\alpha_I\}_{I\subset I_0^{\pm}}$ such that:
	$$F_{\pm} = (\varphi_{\pm}^2)_{I_0^{\pm}}; \:\: f_{\pm} = (\varphi_{\pm}\sqrt{\Lambda_{\pm}})_{I_0^{\pm}}; \:\:
		A_{\pm} = \frac{1}{|I_0^{\pm}|} \sum_{K\subset I_0^{\pm}} \alpha_K^{\pm}(\Lambda_{\pm})^2_K; \:\:
		v_{\pm} = (\Lambda_{\pm})_{I_0^{\pm}},$$
and such that:
	$$\bB_2(F_{\pm}, f_{\pm}, A_{\pm}, v_{\pm}) - \epsilon < \frac{1}{|I_0^{\pm}|} \sum_{I \subset I_0^{\pm}} \alpha_I^{\pm}
	(\varphi_{\pm}\sqrt{\Lambda_{\pm}})^2_{I}.$$
	
We can ``concatenate'' the two trees into a new dyadic tree $T$ centered at $I_0$ -- see Figure \ref{fig:Pic2}(B) --
but, here is the major difference from the usual dyadic situation: $\lambda_{I_0}$, $\varphi(I_0)$ and $\alpha_{I_0}$
must all be assigned to $I_0$, they do not pre-exist. So let some arbitrary $\lambda_{I_0}\geq 0$, $\alpha_{I_0} \geq 0$ and
$\varphi(I_0)\in\bR$. Now we have a new tree $T$, a measure $\Lambda$, a function $\varphi$ and a sequence $\alpha$. 
Next step is to figure out what $F, f, A$,  and $v$ must be, through straightforward calculations.

\begin{align*}
&F = (\varphi^2)_{I_0} = b^2_{I_0} + \frac{1}{2}(F_{-}+F_{+}) \text{, where } b_{I_0} = \frac{\varphi(I_0)}{\sqrt{|I_0|}}\\
&f = (\varphi \sqrt{\Lambda})_{I_0} = a_{I_0}b_{I_0} + \frac{1}{2}(f_{-}+f_{+}) \text{, where } a_{I_0} 
	= \frac{\sqrt{\lambda_{I_0}}}{\sqrt{|I_0|}}\\
& A = \frac{1}{|I_0|} \sum_{K\subset I_0} \alpha_K (\Lambda)^2_K = c_{I_0} + \frac{1}{2}(A_{-}+A_{+}) \text{, where }
	c_{I_0} = \frac{1}{|I_0|}\alpha_{I_0}(\Lambda)^2_{I_0}\\
& v = (\Lambda)_{I_0} = a_{I_0}^2 + \frac{1}{2}(v_{-}+v_+)
\end{align*}

The tree $T$, along with $\Lambda, \varphi$, and $\alpha$ are then admissible for $\bB_2(F, f, A, v)$, and:
	\begin{eqnarray*}
	\bB_2(F, f, A, v) &\geq& \frac{1}{|I_0|} \sum_{I \subset I_0} \alpha_I (\varphi\sqrt{\Lambda})^2_I\\
		&=& \frac{c_{I_0}f^2}{v^2} + \frac{1}{2}\bigg(
			\frac{1}{|I_0^{-}|} \sum_{I\subset I_0^-} \alpha^-_I(\varphi_-\sqrt{\Lambda_-})^2_I
			+ \frac{1}{|I_0^{+}|} \sum_{I\subset I_0^+} \alpha^+_I(\varphi_+\sqrt{\Lambda_+})^2_I\bigg)\\
		&>& \frac{c_{I_0}f^2}{v^2} + \frac{1}{2}\bigg(
		\bB_2(F_-, f_-, A_-, v_-) + \bB_2(F_+, f_+, A_+, v_+) \bigg) - \epsilon.
	\end{eqnarray*}

Therefore, we have the Main Inequality for $\bB_2$:
	\begin{equation}\label{E:MI18}
	c\frac{f^2}{v^2} \leq \bB_2(F, f, A, v) - \frac{1}{2}\bigg(\bB_2(F_{-}, f_{-}, A_{-}, v_{-}) + \bB_2(F_{+}, f_{+}, A_{+}, v_{+}) \bigg),
	\end{equation}
for all quadruplets in the domain of $\bB_2$ such that
	\begin{equation}\label{E:quads}
	F = \widetilde{F} + b^2; \:\: f = \widetilde{f} + ab; \:\: A = \widetilde{A} + c; \:\: v = \widetilde{v} + a^2,
	\end{equation}
and
	$$\widetilde{F} := \frac{F_{-}+F_{+}}{2}; \:\: \widetilde{f} := \frac{f_{-}+ f_{+}}{2}; \:\:
	\widetilde{A} := \frac{A_{-}+A_{+}}{2}; \:\: \widetilde{v} := \frac{v_{-}+v_{+}}{2},$$
and $a\geq 0$, $b \in \mathbb{R}$, $c\geq 0$ are some real numbers.

\begin{lemma}\label{L:BEstimate}
The function $\cB$ in \eqref{E:BellmanDef} satisfies the Main Inequality above in \eqref{E:MI18}.
\end{lemma}

Before we prove this lemma, let us see how it proves Theorem \ref{T:Carleson embedding theorem-tree}.

\begin{proof}[Proof of Theorem \ref{T:Carleson embedding theorem-tree}]

For every $I \in \cD(I_0)$ define:
	\begin{eqnarray*}
	 v_I &:=& (\Lambda)_{I} = \frac{1}{|I|}\lambda_I + \frac{1}{2}(v_{I_{-}}+ v_{I_{+}})
	 	= a_I^2 + \widetilde{v}_I,
	 	\text{ where } a_I := \sqrt{\frac{\lambda_I}{|I|}};\\
	 F_I &:=& (\varphi^2)_I = \frac{1}{|I|}\varphi(I)^2 + \frac{1}{2}(F_{I_{-}}+F_{I_{+}}) 
	 	= b_I^2 +\widetilde{F}_I,
	 	\text{ where } b_I := \frac{\varphi(I)}{\sqrt{|I|}};\\
	 f_I &:=& (\varphi\sqrt{\Lambda})_I = \frac{\varphi(I)\sqrt{\lambda_I}}{|I|} + \frac{1}{2}(f_{I_{-}}+f_{I_{+}}) 
	 	= a_I b_I + \widetilde{f}_I;\\
	 A_I &:=& \frac{1}{|I|}\sum_{K\subset I} \alpha_K (\Lambda)_K^2 = \frac{\alpha_I(\Lambda)_I^2}{|I|} + 
	 	\frac{1}{2}(A_{I_{-}}+ A_{I_{+}})
		 = c_I + \widetilde{A}_I, \text{ where } c_I := \frac{\alpha_I (\Lambda)_I^2}{|I|}.
	\end{eqnarray*}
Note then that
	$$\cB(F_{I}, f_{I}, A_{I}, v_{I}) = \cB(b_I^2 + \widetilde{F}_I, a_Ib_I + \widetilde{f}_I, c_I + \widetilde{A}_I,
		a_I^2 + \widetilde{v}_I),$$
so we may apply Lemma \ref{L:BEstimate} and obtain
	$$\alpha_I f_I^2 \leq |I|\cB(x_I) - |I_{-}|\cB(x_{I_{-}}) - |I_{+}|\cB(x_{I_{+}}).$$
Summing over $I\in\cD(I_0)$ and using the telescoping nature of the sum, we have
	$$\sum_{I\in I_0} \alpha_I f_I^2 \leq |I_0|\cB(F_{I_0}, f_{I_0}, A_{I_0}, v_{I_0}) \leq 4|I_0|F_{I_0},  $$
which is exactly \eqref{E:wCarleson embedding theorem-tree-conc}.
\end{proof}

\begin{remark}
	Before we proceed with the proof of Lemma \ref{L:BEstimate} , let us note that the big and essential difference with Theorem \ref{test1}  now is that in the proof of Theorem \ref{test1} $\{v_I\}_{I\in \cD}$, 
	$\{f_I\}_{I\in \cD}$,
	$\{F_I\}_{I\in \cD}$ are all martingales. This is the standard situation, and it is pictured in Figure \ref{fig:Pic3} (A).	

\begin{figure}[h!]
\centering
\includegraphics[scale=.45]{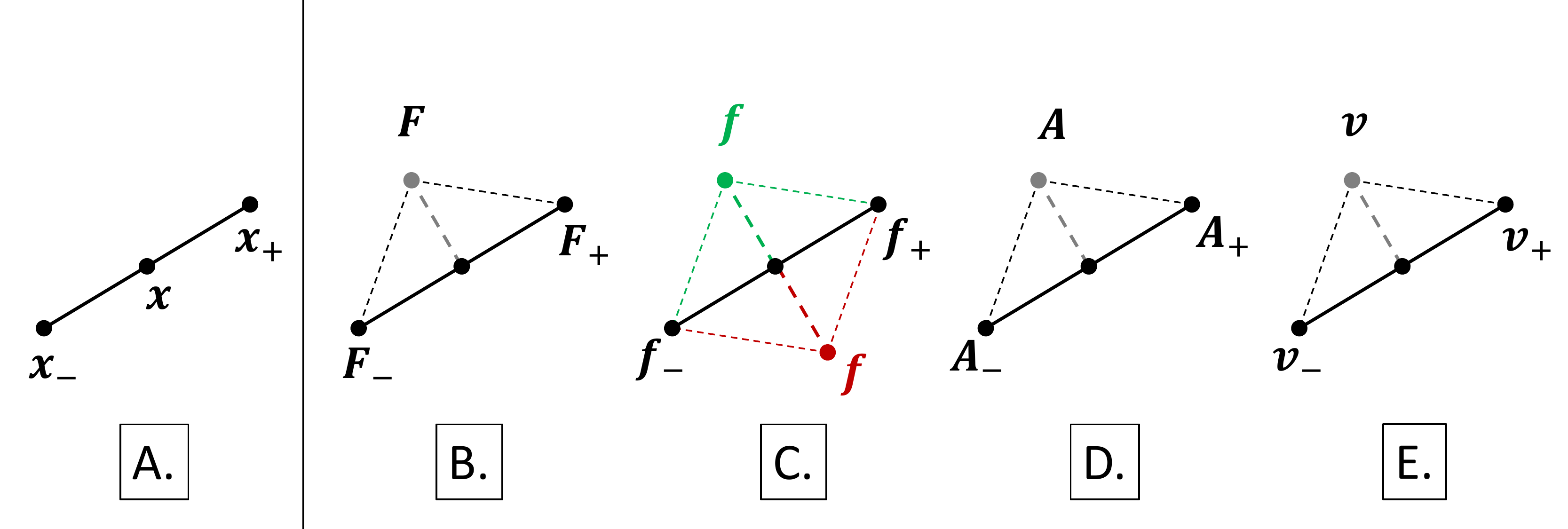}
\caption{Deducing the Main Inequalities for $\bB$}
\label{fig:Pic3}
\end{figure}

	Now looking at \eqref{E:quads}, they  are only \textit{supermartingales} in the case of the $F$, $A$ and $v$ variables,
	and even worse, in the case of $f$ we can have either a supermartingale or a \textit{submartingale}! In other words, we do not have the property \eqref{mart} anymore.	
	Instead, as pictured in Figure \ref{fig:Pic3} (B) -- (E),
	$$
	F_I \ge \frac{F_{I_-}+ F_{I_+}}{2}; \:\: A_I\ge \frac{A_{I_-}+A_{I_+}}{2}; \:\: v_I \ge \frac{v_{I_-}+v_{I_+}}{2},
         $$
 and       
       $$f_I\ge \text{ or } \leq \frac{f_{I_-}+f_{I_+}}{2}.$$

	This is an essential difference, because there is less cancellation, and, indeed, 
	the  fact that $\{f_I\}_{I\in \cD}$ can be both a super or a submartingale can destruct the whole proof. 
	As a small miracle the ``good" supermartingale properties of $\{v_I\}_{I\in \cD}$, $\{F_I\}_{I\in \cD}$ and $\{A_I\}_{I\in\cD}$
	allow us to neutralize the ``bad" sub/supermartingale property of $\{f_I\}_{I\in \cD}$. 	
	The exact calculation of this ``good"--``bad" interplay will be in \eqref{neutr}. 
	We wish to explain now why some variables, the supermartingales 
	$F_I$, $A_I$, $v_I$, are good and some, namely,  $f_I$ is bad.
The explanation is simple: the good ones are those that give positive partial derivative  
of $\cB$, the bad is the one that gives a negative partial derivative	of $\cB$. In fact,
\begin{align*}
&\frac{\pd \cB}{\pd A} \ge 0,\,\,\,\frac{\pd \cB}{\pd v} \ge 0,\,\,\,\frac{\pd \cB}{\pd F}=4,
\\
&\frac{\pd \cB}{\pd f} \le  0\,.
\end{align*}
\end{remark}

\begin{proof}[Proof of Lemma \ref{L:BEstimate}]

Recall that if $g$ is a concave, differentiable function on a convex domain $S\subset \mathbb{R}^n$, then
	$$g(x) - g(x^*) \leq \sum_{i=1}^n \frac{\partial g}{\partial x_i}(x^*) \cdot (x_i - x_i^*),$$
for all $x, x^* \in S$. Denoting $x := (F, f, A, v)$, for the function $\cB$, this takes the particular form:
	\begin{eqnarray}\label{E:GenConcaveEstimate}
	\frac{1}{4}\big(\cB(x) - \cB(x^*)\big) &\leq& (F - F^*) - \frac{2f^*}{v^*+A^*} (f - f^*) \nonumber \\
				&& + \frac{(f^*)^2}{(v^* + A^*)^2} (A - A^*) + \frac{(f^*)^2}{(v^* + A^*)^2} (v - v^*).
	\end{eqnarray}
In particular:
	\begin{equation}\label{E:ConcEstB1}
	\frac{1}{4}\bigg(\cB(F, f, A, v) - \cB(F, f, A-c, v)\bigg) \geq c\frac{f^2}{(v+A)^2} \geq c \frac{f^2}{4v^2},
	\end{equation}
where the last inequality follows because $0 \leq A \leq v$.

By \eqref{E:ConcEstB1}:
	$$c\frac{f^2}{v^2} \leq \bigg(\cB(F, f, A, v) - \cB(\widetilde{F}, \widetilde{f}, A-c, \widetilde{v}) \bigg) +
	\bigg(\cB(\widetilde{F}, \widetilde{f}, A-c, \widetilde{v}) - \cB(F, f, A-c, v) \bigg).$$
We claim that the term in the second parenthesis is negative: apply \eqref{E:GenConcaveEstimate} to obtain
	$$
	\frac{1}{4}\big(\cB(\widetilde{F}, \widetilde{f}, A-c, \widetilde{v}) - \cB(F, f, A-c, v) \big)
	$$
	\begin{align}
	\label{neutr}
	& \leq	-b^2 -\frac{2f}{v+A-c} (-ab) + \frac{f^2}{(v+A-c)^2}(-a^2)
	\\
	&	 -\big(b - \frac{af}{v+A-c} \big)^2 \leq 0.         \notag                          
	\end{align}
Then
	\begin{eqnarray*}
	c\frac{f^2}{v^2} &\leq& \bigg(\cB(F, f, A, v) - \cB(\widetilde{F}, \widetilde{f}, A-c, \widetilde{v}) \bigg)\\
		&\leq& \cB(F, f, A, v) - \frac{1}{2} \bigg(\cB(F_{-}, f_{-}, A_{-}, v_{-}) 
			+ \cB(F_{+}, f_{+}, A_{+}, v_{+}) \bigg),
	\end{eqnarray*}
where $(A_{-}+A_{+})/2 = A - c$, and the last inequality follows by concavity of $\cB$. This proves the lemma.
\end{proof}

\bigskip

\section{Maximal theorem on a tree}
\label{maximal}

Now we are going to prove the result slightly more general than Corollary \ref{C:tree2} from the previous section.

\begin{theorem}
\label{Max:tree}
Let $I_0\in\cD$, the dyadic tree originating at $I_0$ with notations as above. Then, if
	$$\frac{1}{|I|} \sum_{K\subset I} |K|^2 (\Lambda)_K^2 \leq (\Lambda)_I, \quad \forall I \in \cD(I_0),$$
then
	$$\frac{1}{|I_0|} \sum_{I\subset I_0} |I|^2(\Lambda)_I^2 \sup_{K:\, I\subset K}\bigg(\frac{(\varphi\Lambda)_K}{(\Lambda)_K}\bigg)^2 \lesssim (\varphi^2 \Lambda)_{I_0}.$$
\end{theorem}
 
\noindent The proof -- for a change -- is a stopping time proof and not a Bellman proof.
 
 \begin{proof}
For every vertex $H$ of the tree, let us introduce the set of vertices $E_H$. Namely, let $J$ be the first successor of $H$
such that
$$
\frac{(\varphi\Lambda)_J}{(\Lambda)_J } \ge 2 \frac{(\varphi\Lambda)_H}{(\Lambda)_H}\,.
$$
It may happen of course that $J$ is not alone, and there are several first successors with this property.
We call by $E_H$ all vertices that are successors of all these $J'$s and also all such $J's$.

Now we introduce another set of vertices associated with $H$. Consider all successors of $H$ which are not in $E_H$.
All of them plus $H$ itself form the collection that is called $O_H$. This set in never empty (it contains $H$) and can include
 all successors of $H$.
 
 Now we first assign $H=I_0$ and  let $\{J\}$ be the first successors of $H$ with the property above. We call this family stopping 
 vertices of first generation, and denote it by $\mathcal{S}_1$. Then for any $H\in \mathcal{S}_1$ we repeat the procedure thus having stopping vertices of the second generation: $\mathcal{S}_2$.
 
 For each $j$ and each $H\in \mathcal{S}_j$, we have $E_H$ and $O_H$. Notice that all such $O_H$ are disjoint. We call $I_0$ the stopping vertex of $0$ generation, and let $\mathcal{S}= \cup_{j=0}^\infty \mathcal{S}_j$.
 
 Then
 \begin{align*}
 &\sum_{I\subset I_0} |I|^2(\Lambda)_I^2 \sup_{K:\, I\subset K}\bigg(\frac{(\varphi\Lambda)_K}{(\Lambda)_K}\bigg)^2
 \\
 &= \sum_{H\in \mathcal{S}} \sum_{I\in O_H}  |I|^2(\Lambda)_I^2 \sup_{K:\, I\subset K}\bigg(\frac{(\varphi\Lambda)_K}{(\Lambda)_K}\bigg)^2
 \\
 & \leq 4 \sum_{H\in \mathcal{S}}\bigg(\frac{(\varphi\Lambda)_H}{(\Lambda)_H}\bigg)^2  \sum_{I\in O_H}  |I|^2(\Lambda)_I^2
 \\
 & \leq 4 \sum_{H\in \mathcal{S}}\bigg(\frac{(\varphi\Lambda)_H}{(\Lambda)_H}\bigg)^2  \Lambda(H)\,.
 \end{align*}
 The last inequality uses the assumptions of the theorem. But notice that by definition of $E_H$ we easily get
 $$
 \Lambda(E_H) \le \frac{1}{2}\Lambda(H)\Rightarrow \,\, \Lambda(H) \le 2 \Lambda(O_H)\,.
 $$
 Hence
 \begin{equation}\label{E:MaxPf1}
 \sum_{I\subset I_0} |I|^2(\Lambda)_I^2 \sup_{H:\, I\subset H}\bigg(\frac{(\varphi\Lambda)_I}{(\Lambda)_I}\bigg)^2 
   \leq 8 \sum_{H\in \mathcal{S}}\bigg(\frac{(\varphi\Lambda)_H}{(\Lambda)_H}\bigg)^2  \Lambda(O_H).
  \end{equation}
 
 
 Now, define $\beta_H$ for $H \subset I_0$ by $\beta_H := \Lambda(O_H)$ if $H \in \mathcal{S}$, and $\beta_H := 0$ otherwise. Note that, by disjointness of $O_H$, we have
 	\begin{equation*}\label{E:MaxPf2}
	\sum_{H \subset K} \beta_H \leq \Lambda(K), \:\: \forall K \subset I_0.
	\end{equation*}
 Therefore, if we let
 	$$\alpha_H := \frac{\beta_H}{(\Lambda)_H^2}$$
then the sequence $\alpha_H$ satisfies the requirements of the Carleson Embedding Theorem \ref{T:Carleson embedding theorem-tree} for the dyadic tree. So we may rewrite the right hand side of \eqref{E:MaxPf1} in terms of $\beta_H$ and apply Theorem \ref{T:Carleson embedding theorem-tree}:
	\begin{align*}
 	& \frac{1}{|I_0|} \sum_{I\subset I_0} |I|^2(\Lambda)_I^2 \sup_{H:\, I\subset H}\bigg(\frac{(\varphi\Lambda)_I}{(\Lambda)_I}\bigg)^2 
 	\\
 	& \leq 8 \frac{1}{|I_0|} \sum_{H\subset I_0} \alpha_H (\varphi\Lambda)^2_H
 	\\
 	& \leq 32 (\varphi^2\Lambda)_{I_0},
	 \end{align*}
completing the proof of Theorem \ref{Max:tree}.
 
\end{proof}


\section{Two-dimensional version of Theorem \ref{tree1} and dyadic rectangles}

\begin{theorem}\label{t:th3.11}
Let $\mu$ be a positive measure on $R^0 = [0,1)^2$. Let $\langle \mu\rangle_{R}$ denote $\frac{\mu(R)}{|R|}$. Let $\varphi$ be a measurable test function. Then, if
\begin{equation}\label{e:511}
\sum_{Q\subset E,\; Q\in \mathcal{D}(R^0)}\langle \mu\rangle_Q^2|Q|^2 \leq \mu(E),\quad \forall\;E \subset \partial T^2,
\end{equation}
then
\begin{equation}\label{e:521}
\sum_{Q\in \mathcal{D}(R^0)} \langle \varphi \mu\rangle^2_{Q}|Q|^2 \lesssim \langle\varphi^2 \mu\rangle_{R^0}|R^0|.
\end{equation}
\end{theorem}

\begin{remark}
There are now two proofs of this theorem, one in \cite{AMPS} and one in \cite{AHMV2}.
The paper \cite{AMPS} uses capacity and strong capacitary inequalities on the bi-tree, while the proof in \cite{AHMV2} avoids the notion of capacity and strong capacitary estimates completely.
Note that neither the claim nor the conclusion of Theorem \ref{t:th3.11} uses any kind of capacity.
\end{remark}

\begin{remark}
We believe that it would be not enough to check \eqref{e:511} only  for single rectangles. Let us see anyway what we can achieve assuming this one box condition, through a Bellman argument.
\end{remark}

\subsection{One box condition and its corollary}
In the next theorem it is essential to think that $\mu=\{\la_\beta\}$ is the measure on  $\pd T^2$.
Every dyadic rectangle $R$ corresponds to a node of $T^2$, and we will use this in the notations below.
\begin{theorem}\label{t:th3.1}
Let $\mu$ be a positive measure on $R^0 = [0,1)^2$. Let $\langle \mu\rangle_{R}$ denote $\frac{\mu(R)}{|R|}$. Let $\varphi$ be a measurable test function. Then, if
\begin{equation}\label{e:51box}
\sum_{Q\subset R,\; Q\in \mathcal{D}(R^0)}\langle \mu\rangle_Q^2|Q|^2 \leq \mu(R),\quad \forall\;\textup{rectangle}\; R,
\end{equation}
then
\begin{equation}\label{e:52}
\sum_{Q\in \mathcal{D}(R^0)} \langle \varphi \mu\rangle^2_{Q}|Q|^3 \lesssim \langle\varphi^2 \mu\rangle_{R^0}|R^0|.
\end{equation}
\end{theorem}

\begin{proof}
We consider exactly the same function $B(x)$, $x=(F, f, A, v)$,
$$
B(x) = F- \frac{f^2}{v+A}\,.
$$
Given a rectangle $R$ we  consider
$$
F_R =\frac1{|R|}\sum_{\beta\le R}\phi_\beta^2\la_\beta = \frac1{|R|}\int_R \phi^2 d\mu, \, f_R=\frac{1}{|R|} \sum_{\beta\le R} \phi_\beta \lambda_\beta= \frac1{|R|}\int_R \phi d\mu,\,
$$
$$
v_R = \frac1{|R|} \sum_{\beta\le R} \la_\beta=\frac{\mu(R)}{|R|},\, A_R =\frac1{|R|} \sum_{\beta\le R} v_\beta^2 |R_\beta|^2,\,\, x_R=(F_R, \dots, v_R)\,.
$$
Let $R_+, R_-$ be right and left half-rectangles of $R$, and $R^t, R^b$ be top and bottom half-rectangles of $R$ (so, e.g.,  if $R= I\times J$, the $R^t= I\times J_+$).
Now let us estimate from below
$$
B(x_R) -\frac14\bigg( B(x_{R_-}) + B(x_{R_+}) + B(x_{R^t}) +B(x_{R^b})\bigg)\,.
$$
As $\mu$ is concentrated on the boundary, we see immediately, that
$$
F_R= \frac14\bigg(F_{R_-} + F_{R_+} + F_{R^t} + F_{R^b}\bigg),\,\,
f_R= \frac14\bigg(f_{R_-} + f_{R_+} + f_{R^t} + f_{R^b}\bigg)\,.
$$
At the same time,
$$
A_R- \frac14\bigg(A_{R_-} + A_{R_+} + A_{R^t} + A_{R^b}\bigg) \ge \frac1{|R|}\mu(R)^2 ,\,\,
v_R=\frac14\bigg(v_{R_-} + v_{R_+} + v_{R^t} + v_{R^b}\bigg)\,.
$$
The second equality here is just because $v_R=\frac12\bigg(v_{R_-} + v_{R_+} \bigg)$.  And $v_R=\frac12\bigg(v_{R^t} + v_{R^b} \bigg)$.
The first one because    any $\beta$-term in $A_R$  such that this term happens to be in  two rectangles, e.g. in $A_{R_-}$ and $A_{R^t}$, will be cancelled
in the difference. The terms that happen only in one rectangle (this is the case for $R_{--}$ as an example) will be in coefficient $\frac1{|R|}$ in $A_R$, and only with coefficient $\frac1{2|R|}$ in $ \frac14\bigg(A_{R_-} + A_{R_+} + A_{R^t} + A_{R^b}\bigg)$, so it gives a partial (positive) contribution to $A_R- \frac14\bigg(A_{R_-} + A_{R_+} + A_{R^t} + A_{R^b}\bigg)$. And of course, $ \frac1{|R|}\mu(R)^2$ is in $A_R$ and in none of $A_{R_-}, A_{R_+} , A_{R^t}, A_{R^b}$, so it also the part of the contribution.

So we see that three variables $F, f, v$ split in a ``martingale" way, and for $A_R$ we have the above ``super-martingale" inequality.

Thus, considering 
$$
x_R^*= (F_R; f_R; \frac14(A_{R_-} + A_{R_+} + A_{R^t} + A_{R^b}); v_R)
$$
we can write
\begin{align*}
&B(x_R)- B(x_R^*) \ge \frac{\pd B}{\pd A}(x_R)(x_R- x_R^*);
\\
& B(x_R^*)   \ge \frac{1}{4} \bigg( B(x_{R_-})+ B(x_{R_+}) + B(x_{R^t}) +B(x_{R^b})\bigg)\,.
\end{align*}
Here both inequalities are corollaries of the concavity of $B$, in the first one we used that all coordinates of $x_R, x_R^*$ coincide except the $A$-coordinate.
Theerfore, now we  get
$$
B(x_R) -\frac14\bigg( B(x_{R_-}) + B(x_{R_+}) + B(x_{R^t}) +B(x_{R^b})\bigg) \ge c\frac{f_R^2}{v_R^2}\frac{\mu(R)^2}{|R|}
$$
$$
\ge c\frac{\frac1{|R|^2}\big(\int_R\phi d\mu\big)^2}{\mu(R)^2/|R|^2}\frac{\mu(R)^2}{|R|} = \big(\int_R\phi d\mu\big)^2/ |R|\,.
$$
Multiply this by $|R|^2$. We get a term of telescopic sum on the left:
$$
|R|^2B(x_R) -\bigg(|R_-|^2 B(x_{R_-}) + |R_+|^2B(x_{R_+}) + |R^t|^2B(x_{R^t}) +|R^b|^2B(x_{R^b})\bigg) 
$$
$$
\ge c |R| \big(\int_R\phi d\mu\big)^2\,.
$$
Notice that on the next step we pick up all terms $|R| \big(\int_R\phi d\mu\big)^2$ with $R:= R_-, R_+, R^t, R^b$.
Theorem is proved.
\end{proof}

\end{document}